\documentclass[10pt,oneside,reqno]{amsart}
\usepackage{graphicx,tikz-cd} % Required for inserting images
\usetikzlibrary{decorations.pathreplacing}
\usepackage{amssymb,amsmath,amsfonts,amsthm,physics}
\usepackage{subfiles}
\usepackage{hyperref,multicol,multirow,array}
\usepackage[capitalise]{cleveref}
\hypersetup{
    colorlinks=true,
    linkcolor=blue,
    filecolor=magenta,      
    urlcolor=cyan,
    citecolor=blue,
    pdftitle={Spread of a Hamiltonian Path},
    pdfpagemode=FullScreen,
    }
\usepackage[dvipsnames]{xcolor}

\newcolumntype{C}{>{\centering\arraybackslash}m{3cm}}

\newcommand{\evenklein}{%
    \tikz[baseline=-0.5ex, scale=0.35]{
        \draw[fill=lightgray, fill opacity=0.1, line width=1pt]
            (-0.5,0) -- (-0.5,4) -- (3.5,4) -- (3.5,0) -- cycle;
        \draw[line width=1pt, color=PineGreen]
            (2.5,0) -- (0.5,4);
    }%
}
\newcommand{\eventorus}{%
    \tikz[baseline=-0.5ex, scale=0.35]{
        \draw[fill=PineGreen, fill opacity=0.3, line width=1pt]
            (-0.5,0) -- (-0.5,4) -- (3.5,4) -- (3.5,0) -- cycle;
    }%
}

\newcommand{\oddtorus}{%
    \tikz[baseline=-0.5ex, scale=0.35]{
        \draw[fill=lightgray, fill opacity=0.1, line width=1pt]
            (6,0) -- (6,4) -- (10,4) -- (14,0) -- cycle;
        \draw[line width=1pt, color=PineGreen]
            (8,0) -- (8,4);
    }%
}

\newcommand{\oddklein}{%
    \tikz[baseline=-0.5ex, scale=0.35]{
        \draw[fill=PineGreen, fill opacity=0.3, line width=1pt]
            (6,0) -- (6,4) -- (10,4) -- (14,0) -- cycle;
    }%
}
    
\newtheorem{theorem}{Theorem}[section]
\newtheorem{theoremletter}{Theorem}

\newtheorem{proposition}[theorem]{Proposition}
\newtheorem{corollary}[theorem]{Corollary}
\newtheorem{lemma}[theorem]{Lemma}
\theoremstyle{definition}
\newtheorem{definition}[theorem]{Definition}
\newtheorem*{recall}{Recall}
\newtheorem*{remark}{Remark}

\newtheorem{example}[theorem]{Example}

\newtheorem*{question}{Question}

\renewcommand{\dd}{\textup{d}}

\newcommand{\ft}{\mathfrak{t}}

\newcommand{\fg}{\mathfrak{g}}

\newcommand{\IR}{\mathbb{R}}
\newcommand{\IC}{\mathbb{C}}
\newcommand{\IZ}{\mathbb{Z}}
\newcommand{\FF}{\mathbb{F}}

\newcommand{\dR}{\textup{dR}}

\newcommand{\R}{\mathbb{R}}
\newcommand{\C}{\mathbb{C}}

\newcommand{\N}{\mathbb{N}}
\newcommand{\F}{\mathbb{F}}

\newcommand{\vf}[1]{\mathfrak{X}\left( #1 \right)}
\newcommand{\fgz}{\mathfrak{g}_\mathbb{Z}}

\newcommand{\id}{\textup{id}}

\newcommand{\resc}{{\textup{res},c}}

\newcommand{\cut}{\textup{cut}}

\newcommand{\hor}{\textup{hor}}

\newcommand{\Ham}{\textup{Ham}}
\newcommand{\iso}{\textup{iso}}

\title{The Spread Construction and Non-Uniqueness of Visible Lagrangians}

\author{Jo\'e Brendel}  

\address{
D-MATH,
ETH Zürich, 
Rämistrasse 101,
8092 Zürich,
Switzerland }
\email{joe.brendel@math.ethz.ch}

\author{Reto Kaufmann}  

\address{
D-MATH,
ETH Zürich, 
Rämistrasse 101,
8092 Zürich,
Switzerland }
\email{reto.kaufmann@math.ethz.ch}

\begin{document}

\maketitle

\begin{abstract}
    Every even symplectic Hirzebruch surface contains a Lagrangian torus and every odd symplectic Hirzebruch surface contains a Lagrangian Klein bottle as their respective \emph{real locus}. From a toric point of view, this is a visible Lagrangian which surjects to the full moment polytope under the moment map. In this paper, we prove that every Hirzebruch surface contains both a Lagrangian torus and a Lagrangian Klein bottle with this property. An interesting consequence is that the topology of \emph{visible} Lagrangian submanifolds is not determined by their image under the moment map. The proof is based on a new construction, which we call the \emph{spread} of a family of Hamiltonian diffeomorphisms.
\end{abstract}

\tableofcontents
\section{Introduction}
\subsection{Main Result}
Any toric symplectic manifold contains a distinguished Lagrangian submanifold, called \emph{real locus}. A classical example is the real locus of $\C P^n$, which is $\R P^n$. In the general toric case the real locus is the fixed point set of the anti-symplectic involution induced by complex conjugation on $\C^N$ via Delzant's construction. The real locus behaves well with respect to the toric moment map (see \cite{duistermaat1983convexity}). In particular, 
\begin{enumerate}
    \item the image of the real locus under the moment map is the full moment polytope, and
    \item over the interior of the moment polytope, the restriction of the moment map to the real locus is a covering map. 
\end{enumerate} 

The goal of this paper is to show that the real locus is not the only Lagrangian with these properties. In view of this, we refer to any Lagrangian (not necessarily the real locus) satisfying the properties (1) and (2) above as a Lagrangian \emph{covering the moment polytope}. More precisely, what we show is that this property of covering the moment polytope does not even determine the topology of the Lagrangian.

We restrict our attention to symplectic Hirzebruch surfaces $(\F_n,\omega)$, where $n \in \N$. The moment polytopes is a trapezoid in the plane. In the case of $\F_0 = S^2 \times S^2$, the situation is particularly simple: the real locus is just the two-torus obtained as the product of two meridians. In fact, for every even $n$, the real locus is a Lagrangian torus. For odd $n$ it is a Lagrangian Klein bottle. 

\begin{theoremletter}
    \label{thm:mainA}
    Let $(\F_n,\omega)$ be a toric symplectic Hirzebruch surface. There is both a Lagrangian torus and a Lagrangian Klein bottle that cover the moment polytope of $\F_n$. 
\end{theoremletter}

Let us now discuss the main motivation for this result, the study of \emph{visible Lagrangians} \cite[Chapter 5]{evans2023lectures} in toric symplectic manifolds.

\subsection{Visible Lagrangians}
The notion of visible Lagrangians goes back to work of Symington \cite{symington2003fourtwo} who used it for surfaces in (almost) toric symplectic four-manifolds. It can be used to conveniently visualize Lagrangian submanifolds in the Delzant polytopes/almost toric base diagrams of the corresponding ambient space. Interestingly, this visualization method is often sharp in the sense that certain Lagrangians exist if and only if they exist as visible Lagrangians; see Evans--Smith \cite{EvansSmith2018} and Adaloglou--Evans \cite{adaloglou2024klein} for results in that direction for Lagrangian pinwheels and Lagrangian Klein bottles, respectively. For the present paper, we use \cite[Definition 5.2]{evans2023lectures} for Lagrangians in toric manifolds of arbitrary dimension:
\begin{definition}
    Let $L$ be a Lagrangian in a toric symplectic manifold. $L$ is called \emph{visible} if it is compatible with the moment map $\mu$ in the sense that away from the boundary of the polytope, 
\begin{enumerate}
    \item $\mu(L)$ is a submanifold, and
    \item the restriction $\mu\vert_L \colon L \rightarrow \mu(L)$ is a submersion.
\end{enumerate} 
\end{definition}
An extremal case of this scenario is well-known. A fibre of the moment map over the interior of the moment polytope is a Lagrangian torus called a \emph{toric fibre}. This corresponds to the case where the submanifold in the polytope is zero-dimensional and its fibre part is the full fibre. The other extremal case, where the submanifold in the polytope has full dimension, corresponds to what we called a \emph{Lagrangian covering the moment polytope} above. In view of these examples, visible Lagrangians can be viewed as a generalization of both toric fibres and the real locus of a toric manifold.

Visible Lagrangians have a surprising amount of rigidity (see \cite[Theorem 5.1]{evans2023lectures} and the remark thereafter for details). In particular,  
\begin{enumerate}
    \item the image $\mu(L)$ intersects the interior of the moment polytope in an affine linear subspace which is rational with respect to the weight lattice;
    \item the Lagrangian intersects each fibre of the moment map in a disjoint union of tori consisting of translates of the orthogonal complement of the above affine linear subspace.
\end{enumerate}
In view of this and Delzant's classification theorem of toric symplectic manifolds in terms of their moment polytope, it is then very natural to ask the following:
\begin{question} \label{qu:determinedbyimage}
    To what extent is a connected visible Lagrangian determined by its image under the moment map?
\end{question}
For some classes of visible Lagrangians, it is known that the image determines the Lagrangian to a large degree. In case the visible Lagrangian is fixed under a certain type of antisymplectic involution, then \cite[Theorem B]{brendel2023topology} states that its diffeomorphism type is determined by its image under the moment map. Additionally, there is ongoing work by Cannas da Silva--Karshon~\cite{cannas2025toric} which investigates a set of Lagrangians called \emph{toric} which turn out to be the lift of the real locus of a reduction level whose reduced space is itself toric. In that setting, the corresponding Lagrangians are completely classified by their image under the moment map up to \emph{toric isotopies} --- these are a special kind of Hamiltonian isotopies which preserve the fibres of the moment map. Both of these instances are examples of visible Lagrangians. 

In contrast, \cref{thm:mainA} shows that such a conclusion cannot hold for general visible Lagrangians:
\begin{corollary}
    The image under the moment map of a visible Lagrangian $L$ does not determine the topology of $L$.
\end{corollary}

\subsection{The Spread Construction}
The main ingredient to prove \cref{thm:mainA} is a geometric idea using symplectic reduction which we call \emph{spread construction}. Roughly speaking, a smooth family of Hamiltonian diffeomorphisms $\{\varphi^c\}$ on an $S^1$-symplectic quotient can be \emph{spread across} neighboring level sets in the sense that the so-obtained map is again a Hamiltonian diffeomorphism and it acts by $\varphi_c$ on the reduced space at the level $c$. See \cref{thm:mainC} for a precise formulation. Obviously, the neighboring reduced spaces should be sufficiently similar to the one we have started with in an appropriate sense. As this construction may be of independent interest, let us give a detailed outline. 

Let $(M,\omega)$ be a symplectic manifold and $H \colon M \rightarrow \IR$ the moment map of a Hamiltonian $S^1$-action that admits symplectic reduction on the level sets $H^{-1}(c)$ for all $c$ in some segment $(a,b) \subset \IR$. The fact that neighboring reduced spaces are diffeomorphic whenever no singular value is crossed is classical. As the level $H^{-1}(c)$ varies, the symplectic form changes as follows: Assume $0\in (a,b)$ and let $(B,\omega_B)$ be the symplectic quotient at this level. The symplectic forms $\omega_c$ on the reduced spaces $M_c$ vary linearly in $c$, more precisely
\begin{equation*}
    (M_c,\omega_c)\cong (B,\omega_B + c \, F) ,
\end{equation*}
where $F\in \Omega^2(B)$ is the curvature form of some connection on the $S^1$-reduction bundle $H^{-1}(0)\to B$. We refer to the original paper \cite{duistermaat1982variation} by Duistermaat and Heckman for details. The symplectic quotients hence agree up to scaling whenever the curvature is proportional to the symplectic form, that is, if there is $\lambda \in \IR$ such that $F = \lambda \omega_B$, in which case the above simplifies to
\begin{equation*} 
    (M_c,\omega_c) \cong (B,(1+\lambda c)\;\omega_B).
\end{equation*}
A connection with this property exists precisely if the cohomology class $ [\omega_B]\in H^2(B;\IR)$ is a real multiple of the Chern class $c_1(H^{-1}(0)\to B)\in H^2(B;\IZ)$. In view of this, we say that the reduction bundle $H^{-1}(0) \to B$ satisfies the \emph{scaling condition} if there exists $\lambda \in \IR$ such that 
\begin{equation}\label{eq:scaling}
    \lambda [\omega_B] = c_1\left(H^{-1}(0) \to B \right)\in H^2(B;\IZ).
\end{equation}
It follows immediately from the results of Duistermaat and Heckman that if the scaling condition holds at $0$, then it holds at every level $c\in (a,b)$.

Now let $\{\varphi^c \in \Ham(B,\omega_B)\}_{c \in (a,b)}$ be a smooth family of Hamiltonian diffeomorphisms on the model $(B,\omega_B)$ of the symplectic quotients. Every $\varphi^c$ can be viewed as a Hamiltonian diffeomorphism on $(M_c,\omega_c)$, since $\omega_c$ coincides with $\omega_B$ up to scaling by \eqref{eq:scaling}. The \emph{spread} of the family $\{\varphi^c\}$ is a Hamiltonian diffeomorphism $\phi$ of $H^{-1}(a,b) \subset M$ which descends to $\varphi^c$ under symplectic reduction at $c\in (a,b)$. To be precise, we prove the following.
\begin{theoremletter}
    \label{thm:mainC}
    Let $(M,\omega,H)$ be a symplectic manifold equipped with a Hamiltonian $S^1$-action generated by $H$ such that the system admits symplectic reduction at all $H=c$ for $c$ in some segment $(a,b)$, and the scaling condition \eqref{eq:scaling} holds with base manifold $(B,\omega_B)$. For any smooth family $\{\varphi^c \in \Ham(B,\omega_B)\}_{c \in (a,b)}$ of Hamiltonian diffeomorphisms, there is a Hamiltonian diffeomorphism $\phi$ of $H^{-1}(a,b) \subset M$ such that
    \begin{enumerate}
        \item $\phi$ commutes with the $S^1$-action generated by $H$,
        \item the residual Hamiltonian diffeomorphism obtained from $\phi$ on the reduced space $M_c$ is equal, under a suitable identification $M_c = B$, to the prescribed Hamiltonian diffeomorphism $\varphi^c$ for every $c \in (a,b)$.
    \end{enumerate}
\end{theoremletter}

\begin{definition}
    Let $(M,\omega,H)$ and $\{\varphi^c \in \Ham(B,\omega_B)\}_{c \in (a,b)}$ be as in the hypothesis of the above theorem. Then the resulting Hamiltonian diffeomorphism $\phi$ is called the \emph{spread} of the family $\{\varphi^c\}$.
\end{definition}

Here is a simple example where the hypotheses of \cref{thm:mainC} are met. 

\begin{example}
   Let $M = \C^{n+1}$ be equipped with the standard symplectic form and let $H = \pi \vert z_0 \vert^2 + \ldots + \pi \vert z_n \vert^2$. This system admits symplectic reduction for all $c > 0$ and its reduced spaces are symplectomorphic to $\C P^n$ equipped with $c \omega_{\C P^n}$, where $\omega_{\C P^n}$ denotes the Fubini--Study form with normalization $\int_{\C P^1} \omega_{\C P^n} = 1$. The scaling condition is satisfied and thus any family $\{\varphi^c \in \Ham(\C P^n,\omega_{\C P^n})\}_{c \in (a,b)}$ of Hamiltonian diffeomorphisms for $0 < a < b$ can be \emph{spread} across $\C^{n+1}$.
\end{example}

In \cref{sec:spreadconstruction}, we furthermore show that the spread construction is compatible with toric actions and symplectic cuts by the Hamiltonian $H$. This is essential for the proof of \cref{thm:mainA} since we construct the Hirzebruch surfaces using symplectic cuts and rely on the toric structure to identify the resulting space as a Hirzebruch surface.

\subsection{Strategy of the Proof}
\begin{figure}[h]
        \centering
        \begin{tikzpicture}[>=stealth,scale = 0.5, line width=1pt]
            \draw [fill=lightgray, fill opacity= 0.1, line width = 1pt] (-0.5,0) -- (-0.5,4) -- (3.5,4) -- (3.5,0) -- cycle;
            \draw[line width = 1pt, color = PineGreen] (2.5,0) -- (0.5,4) node[midway, sloped, above, color=PineGreen, near end] {\scriptsize Klein} node[midway,sloped,below,color=PineGreen, near end]{\scriptsize bottle};
            
            \draw [fill=lightgray, fill opacity = 0.1, line width = 1pt] (6,0) -- (6,4) -- (10,4) -- (14,0) -- cycle;
            \draw[line width = 1pt, color =  PineGreen] (8,0) -- (8,4) node[midway, sloped, above, color=PineGreen, near end] {\scriptsize Torus};
            
            \draw[line width = 1pt, color=NavyBlue] (-0.5,1) -- (3.5,1) node[color=NavyBlue, above,midway]{\scriptsize Reduction} node[color=NavyBlue, below,midway]{\scriptsize level};
            \draw[line width = 1pt, color=NavyBlue] (6,1) -- (13,1) node[color=NavyBlue, above,midway]{\scriptsize Reduction} node[color=NavyBlue, below,midway]{\scriptsize level};
            
            \fill (8,1) circle[radius=3pt];
            \fill (2,1) circle[radius=3pt];

% Draw the square
\draw[color=NavyBlue] (0,-8) -- (0,-4) node[near start, right]{\footnotesize $S^1$};
\draw (0,-8) -- (4,-8);
\draw (0,-4) -- (4,-4);
\draw (4,-8) -- (4,-4);

% Horizontal arrows (one each)
\draw[->] (2,-8) -- ++(.2,0);   % left to right
\draw[<-] (2,-4) -- ++(-.2,0);  % right to left

% Vertical arrows (two each)
\draw[->] (0,-6) -- ++(0,.2);   % bottom to top
\draw[->] (0,-5.8) -- ++(0,.2); % second arrow lower
\draw[->] (4,-5.8) -- ++(0,.2);  % top to bottom
\draw[->] (4,-6) -- ++(0,.2);% second arrow upper

\draw[color=PineGreen] (0,-8) -- (4,-6) node[midway, above, sloped]{\footnotesize $L$};
\draw[color=PineGreen] (0,-6) -- (4,-4);

% Draw second square
\draw[color=NavyBlue] (6,-8) -- (6,-4) node[near start, right]{\footnotesize $S^1$};
\draw[color=PineGreen] (6,-8) -- (10,-8);
\draw (6,-4) -- (10,-4);
\draw (10,-8) -- (10,-4);

% arrows for second square
\draw[->] (8,-8) -- ++(.2,0);   % left to right
\draw[<-] (8,-4) -- ++(-.2,0);  % right to left

\draw[->] (6,-6) -- ++(0,.2);   % bottom to top
\draw[->] (6,-5.8) -- ++(0,.2); % second arrow lower
\draw[->] (10,-5.8) -- ++(0,.2);  % top to bottom
\draw[->] (10,-6) -- ++(0,.2);% second arrow upper

% rest of H
\draw[color=PineGreen] (6,-6) -- (10,-6) node[midway, above]{\footnotesize $L$};

\draw[-stealth,shorten <= 0.25cm,shorten >= 0.25cm] (2,-4) to[out=60,in=-32] node[midway,right]{\scriptsize $\mu$} (2,1);
\draw[-stealth,shorten <= 0.25cm,shorten >= 0.25cm] (8,-4) to[out=120,in=212] node[midway,left]{\scriptsize $\mu$}(8,1);

\end{tikzpicture}
        \caption{Fibers above the intersection of $\mu(L)$ with a reduction level.}
        \label{fig:Fibers}
\end{figure}
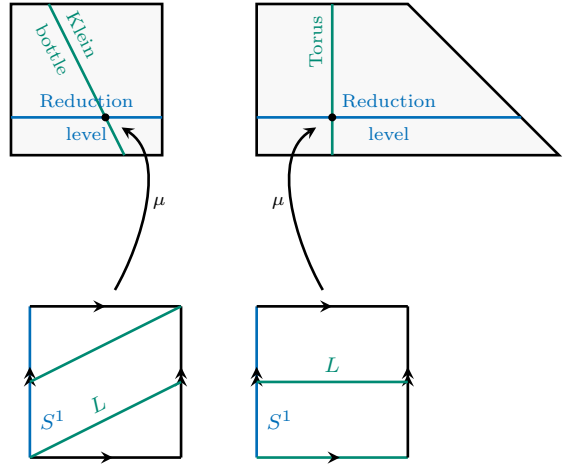
Let $M = (\FF_n,\omega_n)$ be the $n$-th Hirzebruch surface. Then $M$ contains a Lagrangian torus as real locus whenever $n$ is even, and a Lagrangian Klein bottle as real locus whenever $n$ is odd. For concreteness, let us discuss the case of $M=S^2\times S^2$. The toric moment map is simply the product of height functions on the factors. The real locus is given by taking the product of meridians on the sphere factors. Thus in order to prove \cref{thm:mainA} in the case of $S^2\times S^2$, we need to exhibit a Lagrangian Klein bottle which covers the moment polytope. It is also well-known that the Klein bottle can be realized as a visible Lagrangian in $S^2 \times S^2$, equipped with any symplectic form. Indeed, suppose, after possibly exchanging the coordinates of $\IR^2$, that the horizontal side of the rectangle is larger than or equal to the vertical side. Then any segment of slope $(1,-2)$ intersecting the relative interior of the horizontal edges of the rectangle defines a visible Lagrangian Klein bottle, see \cref{fig:Fibers}. See for example \cite[Exercise 5.19]{evans2023lectures} for more details.

\begin{figure}[h]
    \centering
    \begin{tikzpicture}
        
  \shade[ball color = NavyBlue, opacity = 0.2] (-4,0) circle (1cm);
  \draw[color=NavyBlue] (-4,0) circle (1cm);
  \draw[color=NavyBlue] (-5,0) arc (180:360:1 and 0.25);
  \draw[dashed, color=NavyBlue] (-3,0) arc (0:180:1 and 0.25);
  \draw[ color=PineGreen, line width =1pt] (-3.5,-0.86602540378) arc (-90:90:0.15 and 0.86602540378 );
  \draw[dashed, color=PineGreen, line width =1pt] (-3.5,0.86602540378) arc (90:270:0.15 and 0.86602540378 );  
  %\node at (-4,-1.5) {$c=0$};

  \shade[ball color = NavyBlue, opacity = 0.2] (0,0) circle (1cm);
  \draw[color=NavyBlue] (0,0) circle (1cm);
  \draw[color=NavyBlue] (-1,0) arc (180:360:1 and 0.25);
  \draw[dashed, color=NavyBlue] (1,0) arc (0:180:1 and 0.25);
  \draw[dashed, color=PineGreen, line width =1pt] (0,1) arc (90:270:0.25 and 1);
  \draw[ color=PineGreen, line width =1pt] (0,-1) arc (-90:90:0.25 and 1);  
  %\node at (0,-1.5) {$c=0.5$};
   
  \shade[ball color = NavyBlue, opacity = 0.2] (4,0) circle (1cm);
  \draw[color=NavyBlue] (4,0) circle (1cm);
  \draw[color=NavyBlue] (3,0) arc (180:360:1 and 0.25);
  \draw[dashed, color=NavyBlue] (5,0) arc (0:180:1 and 0.25);  
  
 \draw[ color=PineGreen, line width =1pt] (3.5,-0.86602540378) arc (-90:90:0.15 and 0.86602540378 );
  \draw[dashed, color=PineGreen, line width =1pt] (3.5,0.86602540378) arc (90:270:0.15 and 0.86602540378 );
 %\node at (4,-1.5) {$c=1$};

 % Squares above the three spheres
\draw [fill=lightgray, fill opacity= 0.1, line width = 1pt] (-5,3) rectangle (-3,5);   % above (-4,0)
\draw [fill=lightgray, fill opacity= 0.1, line width = 1pt](-1,3) rectangle (1,5);    % above (0,0)
\draw [fill=lightgray, fill opacity= 0.1, line width = 1pt](3,3) rectangle (5,5);     % above (4,0)

% Klein bottle lines inside squares (slope -2)
\draw[color=PineGreen, line width=1pt] (-4,4) ++(-0.5,1) -- ++(1,-2) node[color=PineGreen, above, near start,sloped]{\scriptsize $\mu(L)$};
\draw[color=PineGreen, line width=1pt] (0,4)  ++(-0.5,1) -- ++(1,-2) node[color=PineGreen, above, near start,sloped]{\scriptsize $\mu(L)$};
\draw[color=PineGreen, line width=1pt] (4,4)  ++(-0.5,1) -- ++(1,-2) node[color=PineGreen, above, near end,sloped]{\scriptsize $\mu(L)$};

% Horizontal NavyBlue lines inside squares
\draw[color=NavyBlue, line width=1.5pt] (-5,3) -- (-3,3) node[color=NavyBlue,above,near start]{\scriptsize $\mu_2^{-1}(c)$};   % left: bottom edge
\draw[color=NavyBlue, line width=1.5pt] (-1,4) -- (1,4) node[color=NavyBlue,above,near start]{\scriptsize $\mu_2^{-1}(c)$};    % middle: center
\draw[color=NavyBlue, line width=1.5pt] (3,5) -- (5,5) node[color=NavyBlue,above,midway]{\scriptsize $\mu_2^{-1}(c)$};     % right: top edge

% Vertical arrows with double head at the bottom

\draw[-{>>}] (-4,2.45) -- (-4,1.5);
\draw[-{>>}] (0,3.45)  -- (0,1.5);
\draw[-{>>}] (4,4.45)  -- (4,1.5);

% Labels at uniform height y=2

\node at (-3.7,2) {\scriptsize $/ S^1$};
\node at (0.3,2)  {\scriptsize $/ S^1$};
\node at (4.3,2)  {\scriptsize $/ S^1$};

% Left square
\draw[decorate, decoration={brace, amplitude=10pt, mirror, raise=0.1cm}] 
    (-4-0.9,3) -- (-4+0.9,3);

% Middle square
\draw[decorate, decoration={brace, amplitude=10pt, mirror, raise=0.1cm}] 
    (0-0.9,4) -- (0+0.9,4);

% Right square
\draw[decorate, decoration={brace, amplitude=10pt, mirror, raise=0.1cm}] 
    (4-0.9,5) -- (4+0.9,5);

    \end{tikzpicture}
    \caption{Family of Lagrangians in $B \simeq S^2$ after reduction.}
    \label{fig:Lagrangians}
\end{figure}
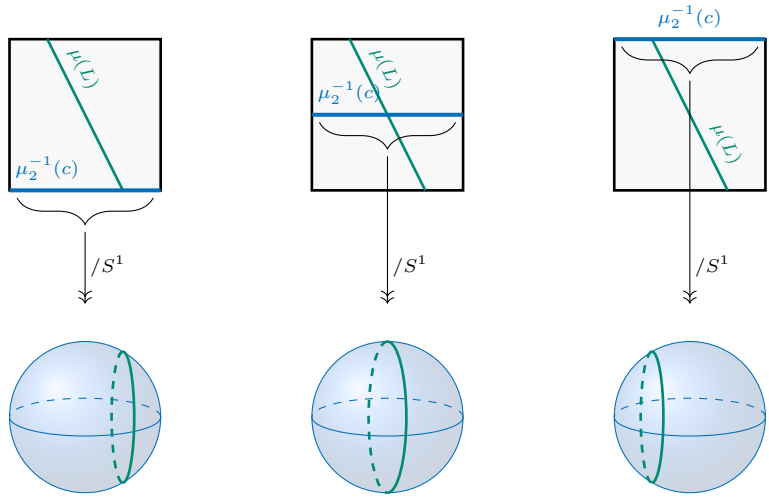

Obviously, this visible Klein bottle does \emph{not} cover the moment polytope. The main idea of the proof is to deform it until it does cover the polytope. To that end, we apply the spread construction and \cref{thm:mainC} to $M=S^2 \times S^2$ and the Hamiltonian $H = \mu_2$ corresponding to the vertical component of the moment map. This makes sense, since the restriction of $\mu_2$ to the visible Klein bottle is already a surjection and the spread allows us to modify the projection of the Klein bottle under the $\mu_1$ component. The reduced spaces $M_c$ can all be identified with the horizontal factor in $S^2 \times S^2$. Since the visible Klein bottle we start with intersects every level of $\mu_2$ in a single point in the rectangle, we find that the projection of the Klein bottle to the reduced space is equal to a circle of constant height in the reduced space, see \cref{fig:Lagrangians}. Note that the circle depends on the level at which we reduce.

\begin{figure}[h]
    \centering
    \begin{tikzpicture}
        
  \shade[ball color = NavyBlue, opacity = 0.2] (-4,0) circle (1cm);
  \draw[color=NavyBlue] (-4,0) circle (1cm);
  \draw[color=NavyBlue] (-5,0) arc (180:360:1 and 0.25);
  \draw[dashed, color=NavyBlue] (-3,0) arc (0:180:1 and 0.25);
  \draw[ color=PineGreen, line width =1pt] (-4.5,-0.86602540378) arc (-90:90:0.15 and 0.86602540378 );
  \draw[dashed, color=PineGreen, line width =1pt] (-4.5,0.86602540378) arc (90:270:0.15 and 0.86602540378 );  
  %\fill[fill=NavyBlue  ,fill opacity=1 ] (-1,0) circle (1pt);
%\fill[fill=NavyBlue  ,fill opacity=1 ] (1,0) circle (1pt);
   
  \shade[ball color = NavyBlue, opacity = 0.2] (0,0) circle (1cm);
  \draw[color=NavyBlue] (0,0) circle (1cm);
  \draw[color=NavyBlue] (-1,0) arc (180:360:1 and 0.25);
  \draw[dashed, color=NavyBlue] (1,0) arc (0:180:1 and 0.25);
  \draw[dashed, color=PineGreen, line width =1pt] (0.86602540378,0.5 ) arc (0:180:0.86602540378 and 0.15);
  \draw[ color=PineGreen, line width =1pt] (-0.86602540378,0.5) arc (180:360:0.86602540378 and 0.15);  
  %\fill[fill=NavyBlue  ,fill opacity=1 ] (-1,0) circle (1pt);
%\fill[fill=NavyBlue  ,fill opacity=1 ] (1,0) circle (1pt);
   
  \shade[ball color = NavyBlue, opacity = 0.2] (4,0) circle (1cm);
  \draw[color=NavyBlue] (4,0) circle (1cm);
  \draw[color=NavyBlue] (3,0) arc (180:360:1 and 0.25);
  \draw[dashed, color=NavyBlue] (5,0) arc (0:180:1 and 0.25);  
  
 \draw[color=PineGreen, line width =1pt,domain=3:5,smooth]   plot (\x,{0.35*(cos(1.2*pi *(\x-4) r)- cos(1.2*pi r))});
 \draw[color=PineGreen, line width =1pt, dashed, domain=3:5,smooth]   plot (\x,{0.45*((cos(0.8 *pi *(\x-4)  r)- cos(0.8*pi r))});
 %\fill[fill=NavyBlue  ,fill opacity=1 ] (-1,0) circle (1pt);
%\fill[fill=NavyBlue  ,fill opacity=1 ] (1,0) circle (1pt);

\draw[-stealth] (-2.75,0) -- node[above]{\scriptsize Rotate} (-1.25,0);
\draw[-stealth] (1.25,0) -- node[above]{\scriptsize Stretch} (2.75,0);
    \end{tikzpicture}
    \caption{Sketch of how to choose the family $\{\varphi^c\}$ on $B$.}
    \label{fig:Family}
\end{figure}
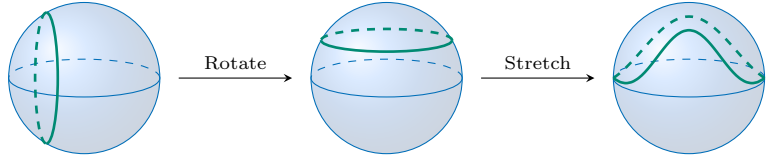
For every $c$, let $\varphi^c$ be a Hamiltonian diffeomorphism of the sphere which maps the circle in question to a Lagrangian circle covering the moment polytope of $S^2$. A visual depiction of such a family is given in \cref{fig:Family}.

Applying the spread construction (\cref{thm:mainC}) to this family yields the claim, since it spreads the Klein bottle in the $\mu_1$-direction and makes it surject to the rectangle. It is not hard to check that the resulting copy of the Klein bottle covers the moment polytope. Again, let us stress that the case of the other Hirzebruch surfaces is similar, see the right-hand side of \cref{fig:Fibers} and \cref{table:overview}.

\begin{table}[h]
\centering

\setlength{\tabcolsep}{8pt}
\renewcommand{\arraystretch}{1.5}

\begin{tabular}{|c|r||C|C|}
\cline{3-4}
\multicolumn{1}{l}{} &
&
\multicolumn{2}{c|}{\textbf{Hirzebruch Surface}}
\\
\cline{3-4}

\multicolumn{1}{l}{}&
&
Even
&
Odd
\\
\hline\hline

\multirow{7.5}{*}{%
    \rotatebox[origin=c]{90}{ \textbf{Visible Lagrangian}
    }%
}
&
\multirow{2}{*}{Torus}
&
Real Locus
&
Spread of
\\
&
&
\eventorus
&
\oddtorus
\\
\cline{2-4}

&
\multirow{2}{*}{Klein Bottle}
&
Spread of
&
Real Locus
\\
&
&
\evenklein
&
\oddklein
\\
\hline

\end{tabular}
\bigskip
\label{table:overview}
\caption{Overview of construction of visible Lagrangians in \cref{thm:mainA}. }
\end{table}

\vspace{0.5cm}

\textbf{Acknowledgements.} We warmly thank Ana Cannas da Silva for her generosity and help. The idea for this project arose out of discussions with her. JB also thanks Ana and the FIM at ETH for hosting a marvelous visit at the start of 2024, during which the groundwork of this article was laid. We thank Felix Schlenk for many helpful comments on an earlier version of this article.

JB is supported by SNSF Ambizione Grant PZ00P2-223460.

\section{Preliminaries about Hamiltonian Isotopies}
The purpose of this section is to collect several standard constructions
relating Hamiltonian isotopies under basic geometric operations that will
appear repeatedly below: rescaling of symplectic forms, symplectic
reduction, symplectic cutting, and reparametrization of Hamiltonian paths.

\medskip

\noindent\textbf{Notation and conventions.} Following \cite{polterovich2012geometry}, we call a function $F:M\times [0,1]\to \IR$ a \emph{Hamiltonian function} on $M$ if it has compact support and write $F_t = F(\cdot, t)$ for all $t\in [0,1].$ Its time-dependent Hamiltonian vector field is denoted by $X^F_t$ and its Hamiltonian isotopy by $\psi^F_t.$ 

\subsection{Rescaling}
\begin{lemma}[Hamiltonian Isotopies and Rescaling] \label{lem:rescalediffeo}
    Let $i=1,2$, $(M_i,\omega_i)$ be symplectic manifolds and $\nu >0$ a real constant. Let $\chi:(M_1,\omega_1) \to (M_2,\omega_2)$ be a diffeomorphism and $F_i:M_i\to \IR$ (possibly time-dependent) Hamiltonians such that $\chi^*(\omega_2) = \nu\;\omega_1 $ and $\chi^*(F_2)=\nu\;F_1.$ Then
    \begin{equation*}
        \chi \circ \psi_t^{F_1} = \psi_t^{F_2} \circ \chi, \qquad \forall t\in [0,1].
    \end{equation*}
\end{lemma}

\subsection{Symplectic Reduction}\label{sec:isotopiesandreduction}
Let $(M,\omega)$ be a symplectic manifold equipped with a Hamiltonian $S^1$-action generated by a moment map $r:M\to \IR.$ Assume that the $S^1$-action is free on the level set $r^{-1}(c)$. We obtain a well-defined symplectic quotient
\begin{equation*}
    \begin{tikzcd}
r^{-1}(c) \arrow[r, "i_c", hook] \arrow[d, "\pi_c", two heads] & {(M,\omega)} \\
{(M_c,\omega_c),}                                                &             
\end{tikzcd}
\end{equation*}
with $i_c^*\omega = \pi_c^*\omega_c.$ 

Assume that $F$ is a (possibly time-dependent) Hamiltonian that Poisson-commutes with the moment map $r$. The condition $\{F_t,r\}=0$ ensures that the Hamiltonian flow of $F$ preserves the level set $r^{-1}(c)$ and that $F_t$ is constant on the $S^1$-orbits for every $t$. As a consequence, the Hamiltonian $F$ descends to a well-defined Hamiltonian $F_\resc:M_c\to \IR$ on the reduced space characterised by 
\begin{equation*}
    i_c^* F = \pi_c^*F_\resc.
\end{equation*}
\begin{definition}
    We call $F_\resc$ the \textit{residual Hamiltonian} on $M_c.$
\end{definition}

\begin{proposition}[Hamiltonian Isotopies and Reduction] \label{prop:residualisotopy}
    Let $(M,\omega,r)$ be a Hamiltonian $S^1$-manifold and assume that $S^1$ acts freely on $r^{-1}(c)$. If $F:M\to \IR$ is a Hamiltonian such that $\{F_t,r\} = 0$ for all $t\in [0,1]$, then 
    \begin{equation*}
        \pi_c \circ \psi^F_t = \psi^{F_{\resc}}_t \circ \pi_c \qquad \forall t\in [0,1]
    \end{equation*}
\end{proposition}

\subsection{Symplectic Cutting}\label{ssec:cutting}
Symplectic cutting, introduced by Lerman \cite{lerman1995symplectic}, is a surgery on Hamiltonian $S^1$-manifolds, which is intimately related to symplectic reduction. The basic observation is that, under some conditions, one can replace a level set $r^{-1}(c)$ of the moment map $r$ by the symplectic quotient $r^{-1}(c)/S^1$ while keeping its superlevel set $\{r > c\}$ unchanged. These two pieces fit together to a smooth manifold equipped with a natural induced symplectic form. 

Let $(M,\omega,r)$ be a Hamiltonian $S^1$-manifold and assume that the $S^1$-action is free on the level set $r^{-1}(c)$. The most convenient way of defining the symplectic cut space is as a symplectic quotient of an auxiliary Hamiltonian $S^1$-space. To that end, let $(M\times\mathbb{C},\omega\oplus\omega_{\mathbb{C}})$ be equipped with the $S^1$-action generated by the moment map
\[
r_\times : M\times\mathbb{C} \to \mathbb{R}, \qquad
r_\times(p,z) = r(p) - \pi |z|^2 - c.
\]
The $S^1$-action is free on $r_\times^{-1}(0)$, and thus symplectic reduction at this level is allowed. 

\begin{definition}
    \label{def:sympcut}
    The \emph{symplectic cut of $(M,\omega)$ at the level $r = c$} is defined as the symplectic quotient 
    \begin{equation*}
        (M_\cut = r_\times^{-1}(0)/S^1,\omega_\cut).
    \end{equation*}
\end{definition}

There are natural symplectic embeddings $\iota_k$ of codimension $2k \in \{0,2\}$, 
\begin{equation*}
    \iota_0 \colon \{r > c\} \hookrightarrow M_\cut, \quad
    \iota_0^* \omega_\cut = \omega,
\end{equation*}
and 
\begin{equation*}
    \iota_1 \colon r^{-1}(c)/S^1 \hookrightarrow M_\cut, \quad
    \iota_1^* \omega_\cut = \omega_c. 
\end{equation*}

Let $F:M\to\mathbb{R}$ be a Hamiltonian satisfying $\{F,r\}=0$. Then the lift $F_{\times}(p,z) = F(p)$ to the auxiliary space $M\times \C$ appearing in Definition~\ref{def:sympcut} descends to a natural Hamiltonian $F_\cut$ on $M_\cut$.
\begin{definition}
    \label{def:cuthamiltonian}
    We call $F_\cut$ the \emph{cut Hamiltonian} associated to $F$.
\end{definition}
The cut Hamiltonian satisfies 
\[
F_\cut\circ\iota_0 = F|_{\{r>c\}}, \qquad
F_\cut\circ\iota_1 = F_\resc,
\]
where $F_\resc$ denotes the residual Hamiltonian on the reduced space $(M_c,\omega_c)$. Obviously $r$ also yields a cut Hamiltonian $r_\cut$ on the cut space.

\begin{proposition}[Hamiltonian Isotopies and Cutting]
\label{prop:cuttingham}
Let $(M,\omega,r)$ be a Hamiltonian $S^1$-manifold such that $S^1$ acts freely on the level set $r^{-1}(c)$. Let $F:M\to\mathbb{R}$ with $\{F,r\}=0$. The Hamiltonian isotopy of $F_\cut$ is then compatible with those of $F$ and $F_\resc$ in the sense that 
    \begin{equation*}
        \psi_t^{F_\cut} \circ \iota_0 =\iota_0 \circ \psi_t^F  \qq{and} \psi_t^{F_\cut} \circ \iota_1=\iota_1 \circ \psi^{F_\resc}_t  \qquad \forall t\in [0,1].
    \end{equation*}
\end{proposition}

\begin{proof}
By the discussion above, we have $\iota_0^* \omega_\cut = \omega$ and $\iota_0^*F_\cut = F$, which proves the first case. The second case follows similarly from $\iota_1^* \omega_\cut = \omega_c$ and $\iota_1^* F_\cut = F_\resc$
\end{proof}

\subsection{Smooth Paths of Hamiltonians}
\begin{proposition}\label{prop:reparametrisation}
Let $(M,\omega)$ be a symplectic manifold and 
$\{\varphi^c:M\to M\}_{c\in[0,1]}$ a smooth path of Hamiltonian
diffeomorphisms. Then there exists a smooth family of Hamiltonians
$\{f^c:M\times[0,1]\to\mathbb{R}\}_{c\in[0,1]}$ such that
\begin{equation*}
\varphi^c = \psi^{f^c}_1 \qquad \forall c\in[0,1].
\end{equation*}
\end{proposition}

\begin{proof}
We define
\begin{equation*}
\varphi_\iso^c := (\varphi^0)^{-1}\circ \varphi^c,
\end{equation*}
such that $\varphi_\iso^0=\id$ and thus $\{\varphi_\iso^c\}$ is a Hamiltonian isotopy.
By Banyaga's Theorem, see also \cite[Proposition 1.4.B]{polterovich2012geometry}, there exists a Hamiltonian
$F:M\times[0,1]\to\mathbb{R}$ such that
\begin{equation*}
\varphi_\iso^c = \psi^F_c .
\end{equation*}
For every $c\in[0,1]$ define the Hamiltonian
\begin{equation*}
f_\iso^c(p,t) := c\,F(p,ct).
\end{equation*}
By the reparametrization property of Hamiltonian flows, see \cite[Exercise 1.4.A]{polterovich2012geometry}, $$\psi^{f_\iso^c}_t = \psi^F_{ct},$$ in particular
$
\psi^{f_\iso^c}_1 = \psi^F_c = \varphi_\iso^c.
$

Choose any Hamiltonian $H:M\times[0,1]\to\mathbb{R}$ such that
$\psi^H_1=\varphi^0$ and define
\begin{equation*}
f^c_t := H_t + f_{\iso,t}^c \circ (\psi^H_t)^{-1}
\end{equation*}
for each $c$. The standard composition formula for Hamiltonian flows implies
\begin{equation*}
\psi^{f^c}_t = \psi^H_t \circ \psi^{f_\iso^c}_t,
\end{equation*}
and hence
\begin{equation*}
\psi^{f^c}_1
= \psi^H_1 \circ \psi^{f_\iso^c}_1
= \varphi^0 \circ \varphi_\iso^c
= \varphi^c. \qedhere
\end{equation*}
\end{proof}

\section{The Spread Construction} \label{sec:spreadconstruction} 
The goal of this section is to define and discuss the spread construction. We start by recollecting facts surrounding the Duistermaat--Heckman normal form theorem for symplectic quotients, before moving to the spread construction and then discussing its compatibility with symplectic cuts. The last subsection is dedicated to lifting torus actions from the base to Duistermaat-Heckman model spaces.

\subsection{The Duistermaat--Heckman Theorem}
The name \emph{Duistermaat--Heckman theorem} usually refers to the statement that the cohomology class of the reduced symplectic form varies linearly with the reduction level. For our purpose, we need the geometric refinement describing the symplectic forms themselves. However, this description depends on the choice of a connection on the reduction bundle. Instead of passing to cohomology, we will impose a condition on this connection. 
\begin{theorem}[Duistermaat--Heckman Theorem for $S^1$-actions]
    Let $(M,\omega,H)$ be a symplectic manifold equipped with a Hamiltonian $S^1$-action generated by $H:M\to \IR$. Assume that $S^1$ acts freely on $H^{-1}(0)$ and let $(B,\omega_B)$ denote the symplectic quotient. Then there exist $a<0<b$ such that $M$ admits a symplectic quotient $(M_c,\omega_c)$ at all $c\in (a,b)$ and $$(M_c,\omega_c)\cong (B,\omega_B + c\, F),$$ where $F\in \Omega^2(B)$ is the curvature form of a connection on the principal circle bundle $H^{-1}(0) \to B.$
\end{theorem}

The proof of the Duistermaat--Heckman theorem proceeds by identifying a neighborhood of the regular level set $H^{-1}(0)$ with a model symplectic manifold whose reduced spaces can be described explicitly. We call this model symplectic manifold the \emph{Duistermaat--Heckman model}.
Its symplectic quotients are easy to compute and exhibit the linear variation of the reduced symplectic forms. In the following proposition, we introduce this model and summarize its main properties. For simplicity, we restrict to the case that we are interested in, that is, we assume that the curvature is proportional to the symplectic form, resulting in a linear scaling of the symplectic form for the quotients.
\begin{proposition}[Duistermaat--Heckman Model]
\label{prop:DHmodel}
Let $(B,\omega_B)$ be a symplectic manifold, $\lambda\in \IR$ a real constant and $Z\to B$ a principal $S^1$-bundle with connection form $\alpha\in \Omega^1(Z)$ whose curvature is $F=\lambda\omega_B.$
Let $I\subset\R$ be an open interval and set $$(M,\omega):=\left(Z\times I,(1+\lambda r)\pi^*\omega_B+\dd r\wedge\alpha\right),$$ where $r:M\to I$ is the projection to the second factor and $\pi:M\to B$ is the composition of the projection $M\to Z$ and the bundle map $Z\to B$. Then:
\begin{enumerate}
\item $\omega$ is symplectic on $\{1+\lambda r\neq0\}$;
\item The induced $S^1$-action on $(M,\omega)$ is Hamiltonian with moment map
$r:M\to \IR$;
\item For every regular value $c$ of $r$ (equivalently $1+\lambda c\neq 0$), the reduced space $(M_c,\omega_c)$ defined by
\begin{equation*}
    \begin{tikzcd}
        r^{-1}(c) \arrow[r,hook,"i_c"] \arrow[d,two heads,"\pi_c"',"/S^1"] & (M,\omega) \\
        (M_c,\omega_c)
    \end{tikzcd}
\end{equation*}
and $\pi_c^*\omega_c = i_c^*\omega$ is a symplectic manifold. There exists a
diffeomorphism $\chi_c:M_c \longrightarrow B$ such that
\begin{equation*}
    \chi_c \circ \pi_c = \pi \circ i_c \qq{and} \omega_c = (1+\lambda c) \chi_c^*\omega_B
.
\end{equation*}
\end{enumerate}
\end{proposition}
\begin{proof}
    \begin{enumerate}
        \item Using that $\dd \alpha =\pi^* F$ and $F=\lambda \omega_B$, we can write
        \begin{equation*}
            (1+\lambda r)\pi^*\omega_B + \dd r\wedge \alpha= \pi^*\omega_B + \dd (r \alpha)
        \end{equation*}
        so $\omega$ is closed. To prove nondegeneracy we use the splitting of the tangent bundle defined by the connection form $\alpha$: If $\partial_\theta\in \mathfrak{X}(Z)$ is the infinitesimal generator of the principal $S^1$-action, then 
        \begin{equation*}
            T(Z\times I) = \ker(\alpha) \oplus \expval{\partial_\theta} \oplus \expval{\partial_r}.
        \end{equation*}
        $\pi^*\omega_B$ is a nondegenerate form on the horizontal part $\ker(\alpha)\cong\pi^*(TB)$, while $\dd r\wedge \alpha$ is nondegenerate on the vertical direction $\expval{\partial_\theta}$ and the $r$-direction. Hence $(1+\lambda r)\pi^*\omega_B +\dd r\wedge \alpha$ is nondegenerate whenever $1+\lambda r\neq 0$.
        \item Let $\partial_\theta\in \mathfrak{X}(Z)$ again denote the infinitesimal generator of the principal $S^1$-action on $Z$. Then
        \begin{align*}
            \iota_{\partial_\theta} \omega = (1+\lambda r)\iota_{\partial_\theta}(\pi^*\omega_B) + \iota_{\partial_\theta}(\dd r\wedge \alpha) =-\dd r,
        \end{align*}
        using that $\pi_*(\partial_\theta) = 0$, $\alpha(\partial_\theta)=1$ and $\dd r(\partial_\theta)=0$. 
        \item Since $\pi\circ i_c:r^{-1}(c)\to B$ is constant on the $S^1$-orbits, there exists a unique smooth map $\chi_c:M_c\to B$ making the diagram
    \begin{equation*}
        \begin{tikzcd}
            r^{-1}(c) \arrow[r, "i_c", hook] \arrow[d, "/S^1","\pi_c"', two heads] & {(M,\omega)} \arrow[d, "\pi", two heads] \\
            {(M_c,\omega_c)} \arrow[r, "\chi_c"]                              & {(B,\omega_B)}                               
        \end{tikzcd}
    \end{equation*}
    commute. The map $i_c$ is a diffeomorphism from $r^{-1}(c)$ to $Z \times \{c\} \subset Z \times I$, and since the $S^1$-action generated by $r$ on $r^{-1}(c)$ coincides with the principal bundle action on $Z$, we deduce that the map $\chi_c$ induced on the respective quotients is a diffeomorphism. 

    Furthermore, we compute
    \begin{align*}
        \pi_c^*\omega_c &= i_c^*((1+\lambda r)\pi^*\omega_B+\dd r\wedge\alpha) \\
        &= (1+\lambda c)(\pi\circ i_c)^*\omega_B \\
        &= (1+\lambda c)(\chi_c \circ \pi_c)^*\omega_B \\
        &= \pi_c^*\left( (1+\lambda c)\chi_c^*\omega_B\right)
    \end{align*}
    and since $\pi_c$ is a surjective submersion, the claim follows. \qedhere
    \end{enumerate}
\end{proof}

\begin{example}
      If $\lambda=0$, $\omega$ is symplectic for all of $I=\IR$ and the Duistermaat--Heckman model is a Hamiltonian $S^1$-space isomorphic to the product $B\times T^*S^1$, equipped with the product symplectic form and the standard Hamiltonian $S^1$-action on the second factor.
\end{example}
\begin{proposition}[Duistermaat--Heckman Normal Form] \label{prop:DHNormalForm}
    Let $(M,\omega,H)$ be a symplectic manifold equipped with a Hamiltonian $S^1$-action generated by $H:M\to \IR$ such that the system admits a symplectic quotient at $0\in \IR$. Then there is a neighborhood of $H^{-1}(0)$ that is equivariantly symplectomorphic to a neighborhood of $Z\times \{0\}$ in the Duistermaat--Heckman model. 
\end{proposition}
\begin{proof}
    This is an application of the coisotropic embedding theorem. We refer to \cite[Section 2.1]{guillemin2012moment} for details.
\end{proof}
\begin{remark}
    It is not hard to see that the neighborhood of $H^{-1}(0)$ can be extended as long as no singular value is crossed. 
\end{remark}

Before proceeding to the main construction, we briefly discuss the assumption that the curvature is proportional to the symplectic form. The following short argument shows that, whenever the cohomology class of the symplectic form is a multiple of the Chern class $c_1(Z)\in H^2(B;\IZ)$, a connection form with this propery exists:
\begin{lemma}[Scaling Condition]\label{lem:prequantization}
    Let $(B,\omega_B)$ be a symplectic manifold and $Z\to B$ a principal $S^1$-bundle with Chern class $c_1(Z)\in H^2(B;\IZ).$ Assume that $$\lambda [\omega_B] = c_1(Z) \qq{for some} \lambda \in \IR.$$ Then there is a connection form $\alpha\in \Omega^1(Z)$ whose curvature $F\in \Omega^2(B)$ satisfies $$F=\lambda \omega_B.$$
\end{lemma}
\begin{proof}
    Pick any connection 1-form $\alpha_0\in \Omega^1(Z)$ and recall that its curvature is defined by $\dd \alpha_0 = \pi^*F_0$. By Chern--Weil theory we have 
    \begin{equation*}
        [F_0]= c_1(Z) = \lambda [\omega_B]
    \end{equation*} and hence there is $\beta\in \Omega^1(B)$ such that $\lambda \omega_B = F_0 +\dd \beta.$ Then 
    \begin{equation*}
        \alpha:= \alpha_0 +\pi^* \beta
    \end{equation*}
    is a connection 1-form such that
    \begin{align*}
        \dd \alpha &= \dd \alpha_0 + \pi^* \dd \beta = \pi^*(F_0 + \dd \beta) = \pi^*(\lambda \omega_B). \qedhere
    \end{align*}
\end{proof}
\begin{remark}
    In the application to Hirzebruch surfaces, all quotients are diffeomorphic to a sphere i.e. $B\cong S^2$. Since $H^2(S^2;\IR)\cong \IR$ is one-dimensional, the hypothesis of \cref{lem:prequantization} is automatically satisfied in this case. 
\end{remark}

\subsection{Proof of \cref{thm:mainC}}
Before starting with the proof, let us give an overview. First, note that, in view of \cref{prop:DHNormalForm}, we can assume without loss of generality that we are working with the Duistermaat--Heckman model $(M,\omega)$ over a base manifold $(B,\omega_B)$. As in \cref{prop:DHmodel}, we denote by $\pi \colon M \rightarrow B$ the natural projection obtained from the symplectic quotient map. To construct the spread of a path of Hamiltonian diffeomorphisms $\{\varphi^c:B\to B\}$, we first apply \cref{prop:reparametrisation} to get a family of Hamiltonians $\{f^c:B\to \IR\}$ whose time-one maps corresponds precisely to the path $\{\varphi^c\}$. We then lift the path $\{f^c\}$ to a single Hamiltonian $F:M\to \IR$ on the total space $M$ and show that the residual Hamiltonians on each reduced space $M_c$ correspond to $f^c$. Hence, its time-one map corresponds to $\varphi^c$. 

The essence of the spread construction for Hamiltonian diffeomorphisms is thus a spread construction for Hamiltonians on the base space of a Duistermaat-Heckman model. So let $(M,\omega)$ be the Duistermaat--Heckman model over the base space $(B,\omega_B)$ as defined in \cref{prop:DHmodel}. In this scenario, the spread is defined as follows: %Recall that the latter is the product $M=Z\times I$ of a principal bundle $\pi:Z\to B$ and an open interval $I\subset \IR$. Recall also that we assume that $Z$ is equipped with a connection on $Z$ such that its curvature form $F\in \Omega^2(B)$ is proportional to $\omega_B$ i.e. there is $\lambda\in \IR$ such that $F=\lambda \omega_B$.
\begin{definition}
     The \textit{spread Hamiltonian} of a smooth path of Hamiltonians $\{f^c:B\to \IR\}_{c\in I}$ is the Hamiltonian defined by
    \begin{align*}
        F:M=Z\times I &\to \IR \\
        (p,c) &\mapsto (1+\lambda c) \cdot f^c(\pi(p,c)),
    \end{align*}
    where $\pi:M\to B$ is the projection of the Duistermaat--Heckman model and $\lambda\in \IR$ is the proportionality factor relating curvature and symplectic form, i.e., $F=\lambda \omega_B$.
\end{definition}

\begin{remark}
    The factor $(1+\lambda c)$ is dictated by the linear variation of the reduced symplectic forms. It compensates exactly for the scaling of $\omega_c$ (see \eqref{eq:scaling}), ensuring that the residual Hamiltonian at level $c$ corresponds to $f^c$ under the identification $M_c \cong B$.
\end{remark}

Recall from \cref{prop:DHmodel} that there is a diffeomorphism $\chi_c:M_c\to B$ such that the diagram 
\begin{equation*}
        \begin{tikzcd}
            r^{-1}(c) \arrow[r, "i_c", hook] \arrow[d, "/S^1","\pi_c"', two heads] & {(M,\omega)} \arrow[d, "\pi", two heads] \\
            {(M_c,\omega_c)} \arrow[r, "\chi_c"]                              & {(B,\omega_B)}                               
        \end{tikzcd}
    \end{equation*} commutes and which satisfies $$\omega_c = (1+\lambda c)\chi_c^*(\omega_B).$$ Recall further from \cref{sec:isotopiesandreduction} that, whenever $\{F,r\}=0$, the Hamiltonian $F$ descends to a residual Hamiltonian $F_\resc:M_c\to \IR$ satisfying 
    \begin{equation}
        i_c^*F = \pi_c^*F_\resc. \label{eq:residualham}
    \end{equation}
\begin{proposition}[Spreading a Path of Hamiltonians] \label{prop:spreadinghamwithDH}
    Let $\{f^c:B \to \IR\}_{c\in I}$ be a smooth path of Hamiltonians and let $F:M \to \IR$ be its spread. Then
    \begin{enumerate}
        \item $\{r,F\} =0$ and
        \item for all $c\in I$, the residual Hamiltonian $F_\resc:M_c \to \IR$ satisfies
        \begin{equation*}
           F_\resc = (1+\lambda c)\chi_c^*(f^c).
        \end{equation*}
    \end{enumerate}
\end{proposition}
\begin{proof}
    \begin{enumerate}
        \item Follows from $F$ being constant on the $S^1$-orbits, which are precisely the fibers of $\pi.$
        \item  By \eqref{eq:residualham}, the definition of the spread $F$ and using that $\pi \circ i_c = \chi_c \circ \pi_c$ we compute
        \begin{align*}
            \pi_c^* F_\resc &= i_c^*F \\
            &= (1+\lambda c) \cdot f^c\circ \pi\circ i_c \\
            &= (1+\lambda c)\cdot f^c\circ \chi_c \circ \pi_c\\
            &= \pi_c^*\left((1+\lambda c) \cdot\chi_c^*f^c \right).
        \end{align*}
        The result follows since $\pi_c^*$ is injective. \qedhere
    \end{enumerate}
\end{proof}

\begin{proof}[Proof of \cref{thm:mainC}]
    By \cref{prop:DHNormalForm}, we can assume without loss of generality that $M$ is the Duistermaat--Heckman model with base manifold $(B,\omega_B)$,  $I=(a,b)$ and $r=H$. By \cref{prop:reparametrisation}, there exists a smooth path of Hamiltonians $\{f^c:B\to \IR\}_{c\in I}$ such that 
    \begin{equation*}
        \varphi^c = \psi_1^{f^c}.
    \end{equation*}
    Let then $F:M \to \IR$ be the spread of this path and note that $\{r,F\}=0$ by \cref{prop:spreadinghamwithDH}. Hence, for all $c\in I$, there is a well-defined residual Hamiltonian $F_\resc:M_c\to \IR$ and, by the main property of the residual Hamiltonian (\cref{prop:residualisotopy}), there is a commutative diagram
    \begin{equation*}
        \begin{tikzcd}
            r^{-1}(c) \arrow[r, "\psi^F_t"] \arrow[d, "\pi_c", two heads] & r^{-1}(c) \arrow[d, "\pi_c", two heads] \\
            M_c \arrow[r, "\psi^{F_\resc}_t"]                             & M_c.                  \end{tikzcd}
    \end{equation*}
    Since $\chi_c:M_c\to B$ is only a rescaling (\cref{lem:rescalediffeo}) and since the Hamiltonians and symplectic forms have the same scaling factor $\frac{1}{1+\lambda c}$ we get the commutative diagram
    \begin{equation*}
        \begin{tikzcd}
M_c \arrow[d, "\chi_c"] \arrow[r, "\psi^{F_\resc}_t"] & M_c \arrow[d, "\chi_c"] \\
B \arrow[r, "\psi^{f^c}_t"]                           & B.                      
\end{tikzcd}
    \end{equation*}
    \cref{thm:mainC} follows by putting the diagrams together and evaluating at $t=1$.
\end{proof}

\subsection{The Spread and Symplectic Cuts} \label{ssec:cutspread}
In the same setup as the previous section, we now \emph{compactify} the space $M = Z \times I$ by performing two symplectic cuts, and show that the spread extends to that set-up. This is not surprising: As discussed in \cref{ssec:cutting}, the symplectic cut operation replaces a level set by its reduced space and the Hamiltonian we spread is lifted from the very same reduced space. 

Let $a_0, b_0 \in I$ with $a_0 < b_0$ and let $(M_\cut,\omega_\cut,r_\cut)$ be the symplectic manifold obtained from $(M, \omega,r)$ by two cuts at $a_0,b_0\in I$ with respect to $r$. Recall from \cref{ssec:cutting} (where this is discussed for a single symplectic cut) that there is a natural symplectomorphism
\begin{equation*}
    \iota_0 \colon \underbrace{r^{-1}(a_0,b_0)}_{\subset M} \rightarrow \underbrace{r_\cut^{-1}(a_0,b_0)}_{\subset M_\cut}, \quad
    \iota_0^* r_\cut = r.
\end{equation*} 
The latter equation tells us that $\iota_0$ is $S^1$-equivariant. Furthermore, there are two symplectomorphisms 
\begin{align*}
    &\iota_1 \colon M_{a_0} \rightarrow r^{-1}_\cut(a_0) \subset M_{\cut}, \\
    &\iota_2 \colon M_{b_0} \rightarrow r^{-1}_\cut(b_0) \subset M_{\cut}.
\end{align*}
We think of $M_\cut$ as decomposing as
\begin{equation*}
    M_\cut 
    = r_\cut^{-1}(a_0) \sqcup r_\cut^{-1}(a_0,b_0) \sqcup r_\cut^{-1}(b_0)
    \cong M_{a_0} \sqcup r^{-1}(a_0,b_0) \sqcup M_{b_0},
\end{equation*}
where the middle space has codimension $0$ and the other two constitute the \emph{cut locus}, which consists of two disjoint codimension two symplectic submanifolds. 

\begin{proposition}
    For every $c \in [a_0,b_0]$, there is a map $\chi_c \colon M_c \rightarrow B$ satisfying
    \begin{equation*}
        \omega_c = (1+\lambda c)\chi_c^* \omega_B
    \end{equation*}
    and making the diagram
    \begin{equation*}
    \begin{tikzcd}
        r_\cut^{-1}(c) \arrow[d] \arrow[r, "i_c", hook] & {(M_\cut,\omega_\cut)} \arrow[d, "\pi_\cut"] \\
        {(M_c,\omega_c)} \arrow[r, "\chi_c"]                                & {(B,\omega_B)}                                
    \end{tikzcd}
    \end{equation*}
    commute. Here, 
    \begin{enumerate}
        \item the vertical map on the right in the diagram is induced on $M_\cut$ by the natural projection $\pi \colon M = Z \times I \rightarrow B$;
        \item in case $c \in (a_0,b_0)$, the vertical map on the left in the diagram is the projection from symplectic reduction;
        \item in case $c \in \{a_0,b_0\}$, the vertical map on the left in the diagram is one of the symplectomorphisms $\iota_1^{-1}, \iota_2^{-1}$. 
    \end{enumerate}
\end{proposition}

\begin{proof}
    Recall that the map $\pi \colon M = Z \times I\to B$ is the composition of the natural projection to the $Z$-component and the projection of $Z$ to the base $B$ of the bundle. It is by construction invariant under the $S^1$-action generated by $r$ and thus descends to a map $\pi_\cut$ on the cut space. This can be checked using the definition of the symplectic cut. 

    For every $c \in (a_0,b_0)$, the $S^1$-equivariant symplectomorphism $\iota_0$ identifies the level set $r^{-1}(c) \subset M$ with $r_\cut^{-1}(c) \subset M_\cut$. This proves that there is a natural identification $M_{\cut,c} = M_c$ and that the proof of the claim follows from \cref{prop:DHmodel}. 

    Now let $c = a_0$ (the case $c=b_0$ works the same) and let $\chi_{a_0} \colon M_{a_0} \rightarrow B$ be the map from \cref{prop:DHmodel} of the reduced space of $M$ (not of $M_\cut$) to $B$. It satisfies $\omega_{a_0}=(1+\lambda a_0)\chi_{a_0}^* \omega_B$. Recall that there is a natural symplectomorphism $\iota_1 \colon M_{a_0} \rightarrow r_\cut^{-1}(c)$. This proves the claim in the case of $c = a_0$. 
\end{proof}

We are now ready to define the spread in case of the cut space. 

\begin{definition}
    Let $\{ f^c:B\to \IR\}_{c\in [a_0,b_0]}$ be a smooth path of Hamiltonians and let $M_\cut$ be as above. The \textit{spread} of this path to $M_\cut$ is the cut Hamiltonian $F:M_\cut \to \IR$, obtained from the spread of any smooth extension of the family $\{f^c\}$ to all of $M = Z \times I$.
\end{definition}

Since the symplectic cut removes the regions $r<a_0$ and $r>b_0$, the cut
Hamiltonian (\cref{def:cuthamiltonian}) is independent of the chosen extension of the
family $\{f^c\}_{c\in[a_0,b_0]}$. The results of the preceding section still hold true in this setting:
\begin{corollary}[Spreading Path of Hamiltonians]
     Let $\{ f^c:B\to \IR\}_{c\in [a_0,b_0]}$ be a smooth path of Hamiltonians and let $F:M_\cut \to \IR$ be its spread. Then
     \begin{enumerate}
         \item $\{r_\cut,F_\cut\}=0 $ and 
         \item for all $c\in [a_0,b_0]$, the residual Hamiltonians $F_\resc:M_c\to \IR$ satisfy
         \begin{equation*}
             F_\resc = (1+\lambda c)\chi_c^*(f^c).
         \end{equation*}
     \end{enumerate}
\end{corollary}
\begin{proof}
    Combine \cref{prop:spreadinghamwithDH} and \cref{prop:cuttingham}.
\end{proof}

\begin{corollary}
    \label{cor:spread_cut}
    Let $\{\varphi^c:B\to B\}_{c\in[a_0,b_0]}$ be a smooth path of Hamiltonian diffeomorphisms. There exists a Hamiltonian $F:M_\cut \to \IR$ such that
    \begin{enumerate}
        \item $\{r_\cut,F\}=0$ and
        \item for all $c \in [a_0,b_0]$ the following diagram commutes:
        \begin{equation*}
            \begin{tikzcd}
r_\cut^{-1}(c) \arrow[r, "\psi^F_1 \vert_{r_\cut^{-1}(c)}"] \arrow[d] & r_\cut^{-1}(c) \arrow[d] \\
M_c \arrow[r, "\psi^{F_\resc}_1"] \arrow[d, "\chi_c"]              & M_c \arrow[d, "\chi_c"]                      \\
B \arrow[r, "\varphi^c"]                                           & B                                           
\end{tikzcd}
        \end{equation*}
    \end{enumerate}
\end{corollary}

\subsection{The Spread and Torus Actions}

In order to identify the symplectic manifold obtained from performing symplectic cuts on a DH-model (see \cref{ex:Hirzebruchs}), we study how to lift Hamiltonian torus actions on the base of $\pi \colon M \rightarrow B$. We focus on toric structures in particular.

Assume now that the base $(B,\omega_B)$ is equipped with a Hamiltonian action by a torus $G$, generated by a moment map $\mu_B:B\to \fg^*.$ In this subsection, we explain when and how this action lifts to a Hamiltonian action on the Duistermaat--Heckman model $(M,\omega)$. Recall that we build the latter from a principal $S^1$-bundle $Z\to B$ with Chern class $c_1(Z)\in H^2(B;\IZ)$ satisfying the scaling condition
\begin{equation}
    \label{eq:scaling2}
    \lambda [\omega_B] = c_1(Z) \qq{for some} \lambda \in \IR.
\end{equation}
This means that $Z$ is a prequantum circle bundle over $(B,\lambda\omega_B)$. We review the relevant conditions for a given action to lift to a prequantum bundle and refer to \cite[Chapter 6]{guillemin2002moment} for details. 
\begin{remark}
    Prequantisation is defined for manifolds $B$ equipped with a closed two--form, not necessarily non-degenerate. The case $\lambda=0$ does therefore not need to be treated separately.
\end{remark}
With regard to the form $\lambda \omega_B$, the moment map is $\lambda \mu_B$. The action then lifts to $Z$ precisely if $\lambda \mu_B$ sends all fixed points to the weight lattice $\fgz^*$ (see \cite[Example 6.10]{guillemin2002moment}). Actually, since $[\lambda\omega_B]$ is integral by the scaling condition, it suffices to assume that there is one fixed point that is mapped to the weight lattice. 
\begin{lemma} \label{lem:hamiltonianlift}
In the situation described above, the Duistermaat--Heckman model $(M,\omega)$ can be equipped with a Hamiltonian $G\times S^1$-action with moment map
$$\mu=((1+\lambda r)\pi^*\mu_B,r):M \to \fg^*\oplus \IR.$$
\end{lemma}
\begin{proof}
    The $G$-action commutes with the principal $S^1$-action and hence $Z$ can indeed be equipped with a $G\times S^1$-action. Since we already know that the $S^1$-action is Hamiltonian with moment map $r$, we are only left to show that the $G$-action is Hamiltonian with moment map $(1+\lambda r)\pi^*\mu_B$. 
    
    The infinitesimal generator for $X\in \fg$ is (compare \cite[Example 6.10]{guillemin2002moment})
\begin{equation*}
    \Tilde{X} =X^\sharp_\textup{hor} -  \pi^*(\lambda \mu_B^X) \cdot \partial_\theta,
\end{equation*}
where $X^\sharp_\hor\in \mathfrak{X}(Z)$ is the horizontal lift of $X^\sharp\in \mathfrak{X}(B)$\footnote{Recall that this means that $X_\hor^\sharp\in \vf{Z}$ is the unique vector field on $Z$ such that \begin{enumerate}
    \item $\pi_*(X^\sharp_\hor)=X^\sharp$ and
    \item $\alpha(X^\sharp_\hor)=0.$
\end{enumerate}} and $\mu_B^X$ is the component of $\mu_B$ along $X$. The result then follows from the computation
\begin{align*}
    \iota_{\Tilde{X}} \omega &= \iota_{\Tilde{X}}\left( (1+\lambda r)\pi^*\omega+ \dd r\wedge \alpha\right) \\
    &= (1+\lambda r)\pi^*(\iota_{X^\sharp}\omega) - \alpha(\Tilde{X}) \dd r \\
    &= (1+ \lambda r) \pi^*(\dd \mu_B^X)  + \pi^*(\lambda \mu_B^X) \,\alpha(\partial_\theta) \dd r \\
    &= (1+\lambda r) \dd (\pi^* \mu_B^X) + \pi^*(\lambda \mu_B^X) \;\dd r \\
    &= \dd \left( (1+\lambda r)\pi^*\mu_B^X\right). \qedhere
\end{align*}
\end{proof}

\begin{corollary}
   In the situation described above, let $\{\varphi^c:B\to B \}_{c\in I}$ be a smooth path of Hamiltonian diffeomorphisms. There exists a Hamiltonian $F:M\to \IR$ such that
   \begin{enumerate}
       \item $\{r,F\} = 0$ and
       \item the following diagram commutes:
       \begin{equation*}
           \begin{tikzcd}
r^{-1}(c) \arrow[r, "\psi^F_1"] \arrow[d, "\pi_c", two heads]      & r^{-1}(c) \arrow[d, "\pi_c", two heads] \arrow[r, hook] & {(M,\omega)} \arrow[d, "(1+\lambda c)\pi^*\mu_B"]   \\
{(M_c,\omega_c)} \arrow[r, "\psi^{F_\resc}_1"] \arrow[d, "\chi_c"] & {(M_c,\omega_c)} \arrow[r, "\mu_c"] \arrow[d, "\chi_c"] & \fg^* \arrow[d, "\frac{1}{1+\lambda c}"] \\
{(B,\omega_B)} \arrow[r, "\varphi^c"]                              & {(B,\omega_B)} \arrow[r, "\mu_B"]                         & \fg^*                               
\end{tikzcd}
       \end{equation*}
   \end{enumerate}
\end{corollary}
\begin{proof}
    The usual spread construction gives the left ladder. The right ladder is commutative by construction of the lifted action.
\end{proof}
\begin{remark}
    This construction is compatible with the cut construction introduced above as long as we cut by the Hamiltonian $r$, which is the second component of $\mu$. Colloquially speaking, this means that we can replace $I$ with any closed or halfclosed interval $I'\subset I$, corresponding to two resp.\ one symplectic cut by the principal $S^1$-action performed on $M$ (see again \cite{lerman1995symplectic}). 
\end{remark}

\begin{example}
    If $\lambda=0,$ $\pi:Z\to B$ is the trivial bundle, there are no restrictions on $I$ and the image of the moment map is 
    \begin{equation*}
        \mu(M) = \bigsqcup_{r\in I}\left(\mu_B(B) \times \{r\}\right) = \mu_B(B) \times I.
    \end{equation*}
    For $I=\IR$, we get $M \cong B\times T^*S^1$, for $I=[0,\infty)$ (i.e. cutting $B\times T^*S^1$ at $r=0$) we obtain $M\cong B\times \IC$, and for $I=[0,1]$ (i.e. cutting at $r=0$ and $r=1$) we have $M \cong B\times S^2.$
\end{example}

\subsubsection*{Toric Actions}
If, in the setting described above, $(B,\omega_B)$ is compact and toric, $[\omega_B]$ is integral exactly when $\mu_B$ maps all fixed points to the weight lattice (see \cite[Exercises 3.6 and 3.7]{guillemin2012moment}). Given the scaling condition \eqref{eq:scaling2}, a toric action on a compact manifold thus always lifts to $Z$ and the lifted action is toric: 

\begin{corollary}
    In the situation described above, with $(B,\omega_B)$  a compact symplectic toric manifold, the Duistermaat--Heckman model $(M,\omega)$ is a toric manifold whose $G\times S^1$-action is generated by the moment map $$\mu=((1+\lambda r)\pi^*\mu_B,r):M \to \fg^*\oplus \IR.$$
\end{corollary}

After compactifying with two cuts as described in \cref{ssec:cutspread}, we obtain a compact symplectic toric manifold, which is thus entirely characterised by the image of its moment map:
\begin{align*}
    \mu(M) &= \left\{ (\varphi,r)\in \fg^*\oplus \IR : r\in I,\; \varphi\in (1+\lambda r)\mu_B(B)\right\} \\
    &= \bigsqcup_{r\in I}\left((1+\lambda r)\mu_B(B) \times \{r\}\right),
\end{align*}
where the second line exhibits $\mu(M)$ as a family of (linearly) scaled copies of $\mu_B(B)$ parametrised by $r$. \smallskip 

Let us now turn to the key example of this section, which will be used later on.

\begin{example} \label{ex:Hirzebruchs}
    Consider $B=S^2$ equipped with the usual symplectic form and moment map given by the height function. If we normalise such that $\mu_B(B) = [0,1]$, $[\omega_B]$ is a generator of $H^2(B;\IZ)\hookrightarrow H^2_\dR(B,\IR).$ The scaling condition \eqref{eq:scaling} thus imposes $\lambda\in \IZ.$ We set $\lambda =-n$ with $n\in \mathbb{N}$ and take $I=[-1,0]$. The moment polytope is then
    \begin{align*}
        \mu(M) &= \bigsqcup_{r\in [-1,0]}(1-nr)[0,1] \times \{r\} \\
        &= \textup{Conv}\{(0,0), (1,0),(0,-1),(1+n,-1)\}
    \end{align*}
    (see \cref{fig:Hirzebruchs}) and we deduce that $M$ is isomorphic to a Hirzebruch-surface $\FF_n.$
    \begin{figure}[h]
        \centering
        \begin{tikzpicture}[scale = 0.3]
            \draw [fill=lightgray, fill opacity= 0.1, line width = 1pt] (-18,0) -- (-18,4) -- (-14,4) -- (-14,0) -- cycle;
            \node at (-16,1) {$n=0$};
            \draw [fill=lightgray, fill opacity = 0.1, line width = 1pt] (-11,0) -- (-11,4) -- (-7,4) -- (-3,0) -- cycle;
            \node at (-8,1) {$n=1$};
            \draw [fill=lightgray, fill opacity = 0.1, line width = 1pt] (0,0) -- (0,4) -- (4,4) -- (12,0) -- cycle;
            \node at (4,1) {$n=2$};
            \draw[-stealth] (17,0) -- node[right]{$r$} (17,4);
        \end{tikzpicture}
        \caption{$\mu(M)$ of \cref{ex:Hirzebruchs} for $n=0,1,2.$}
        \label{fig:Hirzebruchs}
    \end{figure}
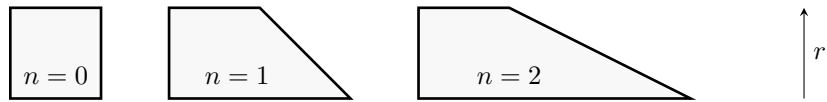
\end{example}

\begin{remark}
    More generally, one can show that if $\mu_B$ is proper, the induced moment map $\mu$ on the Duistermaat--Heckman model is proper as a map to the convex open set $\fg^*\oplus I.$ In view of \cite[Proposition 6.5]{karshon2015non}, we can thus also in this situation identify the resulting toric manifold from the image of its moment map.
\end{remark}

\section{Application to Hirzebruch Surfaces}
\subsection{Setup}
Let $\FF_n$ be the Hirzebruch surfaces constructed in \cref{ex:Hirzebruchs}. We will write $s$ for the horizontal coordinate of the moment map $\mu$, so that $\mu=(s,r)$. Recall in particular that the moment map $\mu$ generates a toric action by $T^2$ and we make the identification $T^2 = S^1 \times S^1$ where $S^1 \times \{1\}:= S$ acts by the flow of $s$ and $\{1\} \times S^1:=R$ acts by the flow of $r$. We will now describe the Lagrangians before the spread, outlining how they can be constructed and what their symmetries are.

Recall that in \cref{ex:Hirzebruchs} we constructed the Hirzebruch surfaces by applying symplectic cuts to the Duistermaat--Heckman model. Before the cuts, the image under the moment map of the Duistermaat--Heckman model looks as sketched in \cref{fig:beforecuts}. We identify the preimage of the interior with $T^2 \times \{\text{interior}\}\subset T^2\times \IR^2$ and the moment map with the projection to the second factor. For any closed subgroup $H\subset T^2$, the conormal bundle $H\times \mathfrak{h}^\perp \subset T^2\times \IR^2$ is a Lagrangian submanifold. Note that this makes sense since we identify $\ft \cong \IR^2 \cong \ft^*$. The Lagrangian condition follows since under this identification we have $\ft^0 = \ft^\perp$. 

Consider then $$H_\text{odd}=S\times \IZ_2\qq{and}H_\text{even}=\{(e^{2i\theta},e^{i\theta}):\theta\in [0,2\pi)\}.$$ This gives Lagrangians over the interior which fiber over line segments as depicted in \cref{fig:beforecuts}. 

\begin{enumerate}
    \item For $H_\text{odd}$, one can think of two cylinders above the line.  It is easy to see that when the cuts are performed, the ends of the two cylinders are glued onto each other and one obtains a visible Lagrangian torus in the Hirzebruch surface. 
    \item For $H_\text{even}$, there is just one cylinder above the line. Each of the symplectic cuts corresponds to adding a cross cap which compactifies the cylinder at one of its ends. Hence we get a visible Klein bottle after the two cuts (see \cite[Exercise 5.16]{evans2023lectures}). 
\end{enumerate}

\begin{remark}
    There is another way to view the torus in (1). By performing symplectic reduction on the segment defined by $\mathfrak{h}_\text{odd}^\perp$ in the Hirzebruch surface (i.e.\ after the symplectic cuts), we obtain a symplectic sphere. This is an example of a so-called \emph{symmetric probe}, see \cite{McD11,AbrBorMcD14}. The projection of the visible torus in the quotient sphere is a meridian. In particular, deforming the meridian to the equator by a Hamiltonian isotopy proves that the torus is Hamiltonian isotopic to a toric fibre in the Hirzebruch surface in question.
\end{remark}

\begin{figure}
    \centering
    \begin{tikzpicture}[>=stealth, scale =0.5]

% Square centered at (-4,0)
\fill[lightgray, opacity=0.1] (-6,-2) rectangle (-2,2);

\draw (-6,-2) -- (-6,2);
\draw (-2,-2) -- (-2,2);

\draw[dotted] (-6,2) -- (-6,3);
\draw[dotted] (-6,-2) -- (-6,-3);

\draw[dotted] (-2,2) -- (-2,3);
\draw[dotted] (-2,-2) -- (-2,-3);

\draw[PineGreen] (-3,-2) -- (-5,2) node[midway, sloped, above]{\scriptsize $\mathfrak{h}_\text{even}^\perp$};
\draw[dotted,PineGreen] (-5,2) -- (-5.5,3);
\draw[dotted,PineGreen] (-3,-2) -- (-2.5,-3);

% Trapezoid with bottom-left at (2,-2)
\fill[lightgray, opacity=0.1] (2,-2) -- (2,2) -- (6,2) -- (10,-2) -- cycle;

% vertical edge
\draw[dotted] (2,-3) -- (2,-2);
\draw (2,-2) -- (2,2);
\draw[dotted] (2,2) -- (2,3);

% sloped edge
\draw[dotted] (5,3) -- (6,2);
\draw (6,2) -- (10,-2);
\draw[dotted] (10,-2) -- (11,-3);

\draw[PineGreen] (4,-2) -- (4,2) node[midway, sloped, above]{\scriptsize $\mathfrak{h}_\text{odd}^\perp$};
\draw[dotted,PineGreen] (4,-3)--(4,-2);
\draw[dotted,PineGreen] (4,2)--(4,3);

\end{tikzpicture}
    \caption{Image of the moment map before symplectic cuts.}
    \label{fig:beforecuts}
\end{figure}
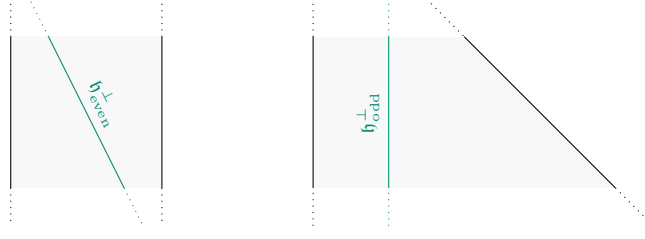

Let $L$ be one of the two Lagrangians constructed above. Visibility of $L$ implies that $r\vert_L:L\to \IR^2$ is transverse to $\{r=c\}$ and hence $r\vert_L^{-1}(c)=L\cap r^{-1}(c)$ is a smooth submanifold. The intersection inherits a free action by
\begin{equation*}
    H_\text{odd} \cap R = \IZ_2 = H_\text{even}\cap R
\end{equation*}
and the quotient $L_c:=(L\cap r^{-1}(c))/\IZ_2$ embeds as a Lagrangian into $M_c$. After rescaling i.e.\ setting $\tilde{L}_c:= \chi_c(L_c)$, this defines a smooth family of Lagrangians in $B=S^2$. For later convenience, we summarise this with the following commutative diagram:
\begin{equation}\label{eq:smoothfamilyofLagrangians}
    \begin{tikzcd}
L\cap r^{-1}(c) \arrow[d, "\pi_{L_c}"',"/\IZ_2", two heads] \arrow[r, hook] & r^{-1}(c) \arrow[d, "\pi_c"',"/S^ 1", two heads]  \\
L_c \arrow[r, hook]   \arrow[d,"\chi_c\vert_{L_c}"]                                                & {(M_c,\omega_c)} \arrow[d,"\chi_c" ] \\
\tilde{L}_c \arrow[r,hook] & (B,\omega_B).
\end{tikzcd}
\end{equation}

\subsection{Proof of \cref{thm:mainA}}
\begin{proof}[Proof of \cref{thm:mainA}]
    Let $\{\varphi^c:B \to B\}_{c\in [0,1]}$ be a smooth path of Hamiltonian diffeomorphisms of the standard sphere such that $\varphi^c(\tilde{L}_c)\subset B$ is a Lagrangian covering the moment polytope $\mu_B(B)$ for every $c \in [0,1]$ (see also \cref{fig:Family}). Let $F:\FF_n\times [0,1] \rightarrow \R$ be its spread to $\FF_n$ and $\psi^F_1$ the Hamiltonian diffeomorphism generated by $F$. By the properties of the spread construction, we get
    \begin{enumerate}
        \item $r \circ \psi^F_1 =r$ and
        \item the commutative diagram for all $c \in [0,1]$
        \begin{equation} \label{eq:centralladder}
            \begin{tikzcd}
        r^{-1}(c) \arrow[r, "\psi^F_1"] \arrow[d, "\pi_c", two heads]      & r^{-1}(c)  \arrow[d, "\pi_c", two heads]            \\
        {(M_c,\omega_c)} \arrow[r, "\psi^{F_\resc}_1"] \arrow[d, "\chi_c",two heads] & {(M_c,\omega_c)} \arrow[d, "\chi_c", two heads]  \\
        {(B,\omega_B)} \arrow[r, "\varphi^c"]                              & {(B,\omega_B)}.                                                        
    \end{tikzcd}
        \end{equation}
         We recall that $(M_0,\omega_0)$ and $(M_1,\omega_1)$ are not reduced spaces, but rather the cut locus of the two symplectic cuts, as in Section \ref{ssec:cutspread}. More specifically, see \cref{cor:spread_cut}.
    \end{enumerate}

For every $c \in [0,1]$, let us analyse what the image of $L \cap r^{-1}(c)$ under $s$ looks like after the spread. The map $r$ is preserved under the isotopy of the spread. To prove surjectivity of the restricted moment map $\mu\vert_{\psi_1^F(L)}$ it is thus sufficient to show that $s(\psi^F_1(L \cap r^{-1}(c)))$ is equal to the image of the full level set $r^{-1}(c)$. 

First, note that the commutative diagram \eqref{eq:centralladder} can be completed on the left with diagram \eqref{eq:smoothfamilyofLagrangians}, summarising the construction of the smooth family $\{\tilde{L}_c\}_c$. Second, the right side can be complemented with the diagrams describing the behaviour of the moment map under reduction and rescaling respectively, giving
\begin{equation*}
    \begin{tikzcd}
        L\cap r^{-1}(c) \arrow[r, hook] \arrow[d, "\pi_{L_c}", two heads] & r^{-1}(c) \arrow[r, "\psi^F_1"] \arrow[d, "\pi_c", two heads]      & r^{-1}(c) \arrow[r, hook] \arrow[d, "\pi_c", two heads] & {(M,\omega)} \arrow[d, "s"]            \\
        L_c \arrow[r, hook] \arrow[d, "\chi_c\vert_{L_c}",two heads]                & {(M_c,\omega_c)} \arrow[r, "\psi^{F_\resc}_1"] \arrow[d, "\chi_c",two heads] & {(M_c,\omega_c)} \arrow[r, "\mu_c"] \arrow[d, "\chi_c", two heads] & \IR \arrow[d, "\frac{1}{1+\lambda c}"] \\
        {\tilde{L}_c} \arrow[r, hook]                                         & {(B,\omega_B)} \arrow[r, "\varphi^c"]                              & {(B,\omega_B)} \arrow[r, "\mu_B"]                       & \IR.                                   
    \end{tikzcd}
    \end{equation*}

Chasing the diagram, we see that
\begin{align*}
    \frac{1}{1+\lambda c} s\left( \psi_1^F(L\cap r^{-1}(c))\right) &= \mu_B(\varphi^c(\tilde{L}_c)) \\
    &= \mu_B(B) \\
    &= \mu_B(\chi_c(\pi_c(r^{-1}(c)))) \\
    &= \frac{1}{1+\lambda c}s(r^{-1}(c)),
\end{align*}
and conclude that
\begin{equation}
    \label{eq:s_restriction}
    s(\psi^F_1(L\cap r^{-1}(c))) = s(r^{-1}(c)).
\end{equation}
We stress again that this holds for all $c \in [0,1]$, boundary points included.

Let us now prove that $\mu\vert_{\psi_1^F(L)} = (s\vert_{\psi_1^F(L)},r\vert_{\psi_1^F(L)})$ is a submersion. To that end, note that differentiating $r \circ \psi_F^1 = r$ and \eqref{eq:s_restriction} yields
\begin{align*}
    dr\vert_{\psi_1^F(L)} &= dr ;\\
    ds\vert_{\psi_1^F(L) \cap r^{-1}(c)} &= ds\vert_{r^{-1}(c)}.
\end{align*}
The RHS terms are obviously linearly independent, proving the claim.
\end{proof}
\begin{remark}
    Since the spread construction preserves the $R$-action, the final Lagrangian $\psi_ 1^F(L)$ is invariant under $H\cap R\cong \IZ_2.$
\end{remark}

\begin{remark}
    Note that the above proof actually shows that also the edges are covered by $\psi^F_1(L)$. More precisely, the differential of $s\vert_{\psi^F_1(L)}$ does not vanish above the interior of the horizontal edges. Similarly, the differential of $r\vert_{\psi^F_1(L)}$ does not vanish over the remaining two edges.
\end{remark}

\bibliographystyle{plain}
\bibliography{references}

\end{document}